\UseRawInputEncoding
\documentclass[11pt,a4paper]{article}
\usepackage{}
\usepackage{bbding}
\usepackage{amsfonts}
\usepackage{amssymb}
\usepackage{mathrsfs}
\usepackage{amsmath}
\usepackage{epsf,epsfig,amsgen,amstext,amsbsy,amsopn,amsthm,indentfirst}
\usepackage{lineno}%显示数字导言区
\allowdisplaybreaks[4]
\usepackage{multicol,indentfirst}
\usepackage{subfigure}

\newtheorem{theorem}{Theorem}[section]

\newtheorem{lemma}{Lemma}[section]

\begin{document}
%\linenumbers %显示数字
\title
{\LARGE \textbf{The characterizing properties of (signless) Laplacian permanental polynomials of bicyclic graphs}}

\author{ Tingzeng Wu\thanks{{Corresponding author.\newline
\emph{E-mail address}: mathtzwu@163.com, zhoumaths19@163.com}}, Tian Zhou\\
{\small School of Mathematics and Statistics, Qinghai Nationalities University, }\\
{\small  Xining, Qinghai 810007, P.R.~China}}

\date{}

\maketitle

\noindent {\bf Abstract: } Let $G$ be a graph with $n$ vertices, and let $L(G)$ and $Q(G)$ be the Laplacian matrix and signless Laplacian matrix of $G$, respectively. The polynomial $\pi(L(G);x)={\rm per}(xI-L(G))$ (resp. $\pi(Q(G);x)={\rm per}(xI-Q(G))$) is called {\em Laplacian permanental polynomial} (resp. {\em signless Laplacian permanental polynomial}) of $G$. In this paper, we  show that two classes of bicyclic graphs are determined by their (signless) Laplacian permanental polynomials.

\section{Introduction}

Let $G$ be a simple and undirected graph with vertex set $V(G)=\{v_{1}, v_{2},..., v_{n}\}$ and edge set $E(G)$. Denote the degree of a vertex $ v_{i}$ by $d(v_{i})$.
The cycle, path and complete graph on $n$ vertices are denoted by $C_{n}$, $P_{n}$ and $K_{n}$, respectively.
We give the definitions of  three types of special graphs which are used later.
\begin{itemize}
\item Let $D_{r,n-r}$ be a graph obtained by joining an edge between a vertex of the cycle $C_{r}$ ($3\leq r\leq n-1$) and a pendant vertex of the path $P_{n-r}$ ($n\geq4$).

\item Let $d(p, q, r)$ be a bicyclic graph on $n$ vertices obtained by identifying $C_{p}$ ($p\geq3$) and $C_{q}$ ($q\geq3$) with two different end vertices of $P_{r}$. See Figure \ref{fig1}.

\item Let $\theta(p, q, r)$  be a bicyclic graph with $n$ vertices consisting of two given vertices joined by three disjoint paths whose order are $p$, $q$ and $r$, respectively, where $p\geq 0$ , $q\geq 0$ and $r \geq 0$, and at most one of them is 0. See Figure \ref{fig1}.

\end{itemize}

The {\em permanent} of $n\times n$ matrix $X=(x_{ij})(i,j=1,2,\ldots,n)$ is defined as $${\rm per}(X)=\sum_{\sigma}\prod_{i=1}^{n}x_{i\sigma(i)},$$
where the sum is taken over all permutations $\sigma$ of $\{1,2,\ldots, n\}$. Valiant \cite{val} has shown that computing the permanent is \#P-complete even when restricted to (0, 1)-matrices.

Let $G$ be a graph, and let $D(G)=diag(d(v_{1}), d(v_{2}),..., d(v_{n}))$ be the diagonal matrix of vertex degrees of $G$. The
{\em Laplacian matrix} is $L(G)=D(G)-A(G)$ and {\em signless Laplacian matrix} is $Q(G)=D(G)+A(G)$, where $A(G)$ is the $(0, 1)$-adjacency matrix.
The polynomial
\begin{eqnarray*}\label{equ1}
\pi(L(G);x)={\rm per}(xI-L(G)) (\text{resp.}  \pi(Q(G);x)={\rm per}(xI-Q(G)))
 \end{eqnarray*}
is called  {\em Laplacian permanental polynomial} (resp. {\em signless Laplacian permanental polynomial}) of $G$, where $I$ is the identity matrix of size $n$.

Graphs $G$ and $H$ are called {\em   Laplacian copermanental} if $\pi(L(G);  x)=\pi(L(H); x)$. Similarly, if $\pi(Q(G); x)=\pi(Q(H); x)$ then $G$ and $H$ are {\em  signless Laplacian copermanental}. We say a graph $G$ is determined by its (signless) Laplacian permanental polynomial if any graph (signless) Laplacian copermanental with $G$ is isomorphic to $G$.

The Laplacian permanental polynomials of graphs were first studied by Merris et al. \cite{mer}. The studies on
 Laplacian permanental polynomial mainly focus on two aspects, one is computing the coefficients of Laplacian permanental polynomial of a graph \cite{bap,cas,gen,gen2, lis, liu, liuwu1, mer2, vrb}. The other is distinguishing graphs by the Laplacian permanental polynomial. Merris et al. \cite{mer} first considered the problem which graph is determined by its Laplacian permanental polynomial. They stated that they do not know of a pair of
nonisomorphic Laplacian copermanental graphs. And they found that there exist no Laplacian copermanental graphs for all connected graphs with 6 vertices. Furthermore, Merris and collaborators \cite{bot} proved that no two trees are Laplacian copermanental. Recently, Liu \cite{liu} showed that complete graphs and complete bipartite graphs are determined by their Laplacian permanental polynomials. Liu and Wu \cite{liuwu2} proved that path, cycle and lollipop graph are determined by their Laplacian permanental polynomials.

Faria \cite{far} first considered the signless Laplacian permanental polynomial of $G$. And he found that $\pi(Q(G);x)=\pi(L(G);x)$ when $G$ is a bipartite graph. Furthermore, he discussed the multiplicity of interger roots of  $\pi(Q(G);x)$ \cite{far2}. Up to now,
only a few results have been obtained on the signless Laplacian permanental polynomials. Li and Zhang \cite{lizh,lizh2} gave  the bounds of constant terms of signless Laplacian permanental polynomials of some graphs. Liu \cite{liu} showed that complete graphs and complete bipartite graphs are determined by their signless Laplacian permanental polynomials. Liu and Wu \cite{liuwu2} proved that path, cycle and lollipop graph are determined by their signless Laplacian permanental polynomials. Recently, Wu et al. \cite{wut} proved that $d(3, 3, r)$ are determined by its signless Laplacian permanental polynomial. And they proposed a problem as follows.

\noindent{\bf Problem:} Prove that graphs $d(p, q, r)$ and $\theta(p, q, r)$  are determined by their (signless) Laplacian permanental polynomials.

In this paper, we mainly focus on the problem. And we give the solution to the problem as follows.
\begin{theorem}\label{art11}
Graph $d(p, q, r)$ is determined by its Laplacian permanental polynomial.
\end{theorem}
\begin{theorem}\label{art12}
Graph $\theta(p, q, r)$ is determined by its Laplacian permanental polynomial.
\end{theorem}
\begin{theorem}\label{art13}
Graph $d(p, q, r)$ is determined by its signless Laplacian permanental polynomial.
\end{theorem}
\begin{theorem}\label{art14}
Graph $\theta(p, q, r)$ is determined by its signless Laplacian permanental polynomial.
\end{theorem}
% An outline of this paper is as follows. In Section 2, we shall present some  lemmas and corollaries. In Section 3, we prove that two classes of bicyclic graphs are determined by their  signless Laplacian permanental polynomials.
 \begin{figure}[htbp]
\begin{center}
\includegraphics[scale=0.6]{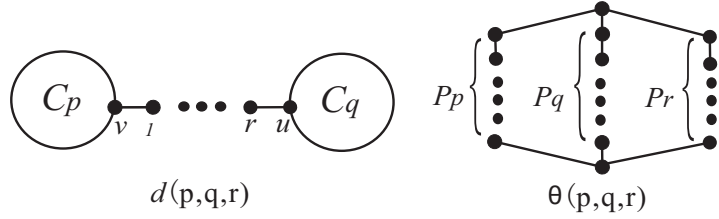}
\caption{\label{fig1}\small
{Bicyclic graphs $d(p,q,r)$ and $\theta(p,q,r)$.}}
\end{center}
\end{figure}

\section{Some Lemmas}
\begin{lemma}\label{art21}{\cite{liu}}
Let $G$ be a graph with $n$ vertices and $m$ edges, and let $(d_{1}, d_{2},..., d_{n})$
be the degree sequence of $G$. Suppose that $\pi(L(G);x)=\sum\limits_{i}p_{i}(G)x^{n-i}$. Then\\
\noindent
$(i)$ $p_{0}(G)=1$;\\
$(ii)$ $p_{1}(G)=-2m$;\\
$(iii)$ $p_{2}(G)=2m^{2}+m-\frac{1}{2}\sum\limits_{i}d^{2}_{i}$; \\
$(iv)$ $p_{3}(G)=-\frac{1}{3}\sum\limits_{i}d^{3}_{i}+
(m+1)\sum\limits_{i}d^{2}_{i}-\frac{4}{3}m^3-2m^2+2c_{3}(G)$;\\
$(v)$ $p_{4}(G)=-\frac{1}{4}\sum\limits_{i}d^{4}_{i}+(\frac{2}{3}m+1)
\sum\limits_{i}d^{3}_{i}-\frac{1}{2}(2m^2+5m+1)\sum\limits_{i}d^{2}_{i}
+\frac{1}{8}(\sum\limits_{i}d^{2}_{i})^2+\sum\limits_{(v_{i}v_{j})\in
E(G)}d_{i}d_{j}+2\sum\limits_{i}d_{i}c_{3}(G_{v_{i}})+2c_{4}(G)-4mc_{3}(G)+\frac{2}{3}m^4+2m^3+\frac{1}{2}m^2
+\frac{1}{2}m$.
\end{lemma}

\begin{lemma}\label{art22}{\cite{liu}}
Let $G$ be a graph with $n$ vertices and $m$ edges, and let $(d_{1}, d_{2},..., d_{n})$
be the degree sequence of $G$. Suppose that $\pi(Q(G);x)=\sum\limits_{i}q_{i}(G)x^{n-i}$. Then\\
\noindent
$(i)$ $q_{0}(G)=1$;\\
$(ii)$ $q_{1}(G)=-2m$;\\
$(iii)$ $q_{2}(G)=2m^{2}+m-\frac{1}{2}\sum\limits_{i}d^{2}_{i}$; \\
$(iv)$ $q_{3}(G)=-\frac{1}{3}\sum\limits_{i}d^{3}_{i}+
(m+1)\sum\limits_{i}d^{2}_{i}-\frac{4}{3}m^3-2m^2-2c_{3}(G)$;\\
$(v)$ $q_{4}(G)=-\frac{1}{4}\sum\limits_{i}d^{4}_{i}+(\frac{2}{3}m+1)
\sum\limits_{i}d^{3}_{i}-\frac{1}{2}(2m^2+5m+1)\sum\limits_{i}d^{2}_{i}
+\frac{1}{8}(\sum\limits_{i}d^{2}_{i})^2+\sum\limits_{(v_{i}v_{j})\in
E(G)}d_{i}d_{j}-2\sum\limits_{i}d_{i}c_{3}(G_{v_{i}})+2c_{4}(G)+4mc_{3}(G)+\frac{2}{3}m^4+2m^3+\frac{1}{2}m^2
+\frac{1}{2}m$.
\end{lemma}

\begin{lemma}\label{art23}{\cite{liuwu2}}
The following can be deduced directly from Laplacian permanental polynomial of a graph $G$:\\
(i)The vertices of graph $G$;\\
(ii)The edges of graph $G$;\\
(iii)The sum of square of degrees of graph $G$;\\
(iv)$\sum\limits_{i}d^{3}_{i}-6c_{3}(G)$,  where $c_{3}(G)$ is the number of the triangles in $G$ and $d_{i}$ is the degree of the vertex $v_{i}$ of $G$.
\end{lemma}

\begin{lemma}\label{art24}{\cite{liuwu2}}
The following can be deduced directly from signless Laplacian permanental polynomial of a graph $G$:\\
(i)The vertices of graph $G$;\\
(ii)The edges of graph $G$;\\
(iii)The sum of square of degrees of graph $G$;\\
(iv)$\sum\limits_{i}d^{3}_{i}+6c_{3}(G)$,  where $c_{3}(G)$ is the number of the triangles in $G$ and $d_{i}$ is the degree of the vertex $v_{i}$ of $G$.
\end{lemma}

\begin{lemma}\label{art25}{\cite{liuwu1}}
 Let $v$ be a vertex of a graph $G$, $\mathscr{C}_{G}(v)$ the set of cycles of $G$ containing $v$ and $N(v)$ the set of vertices of $G$
adjacent to $v$. Then\\
$(a)$
$\pi(L(G);x)=(x-d(v))\pi(L_{v}(G))+\sum\limits_{u\in N(v)}\pi(L_{uv}(G))+2\sum\limits_{C\in \mathscr{C}_{G}(v)} \pi
(L_{V(C)}(G));$\\
$(b)$ $\pi(Q(G);x)=(x-d(v))\pi(Q_{v}(G))+\sum\limits_{u\in N(v)}\pi(Q_{uv}(G))+2\sum\limits_{C\in \mathscr{C}_{G}(v)}(-1)^{\mid V(C)\mid} \pi
(Q_{V(C)}(G)).$
\end{lemma}

\begin{lemma}\label{art26}{\cite{liuwu1}}
Let $G$ and $H$ be vertex-disjoint graphs. Let $G\cdot H$ be the coalescence obtained from $G$ and $H$ by identifing a vertex $u$ of $G$ with a vertex $v$ of $H$. Then\\
$(a)$
$\pi(L(G\cdot H))=\pi(L(G))\pi(L_{v}(H))+\pi(L(H))\pi(L_{u}(G)-x\pi(L_{v}(H))\pi(L_{u}(G);$\\
$(b)$ $\pi(Q(G\cdot H))=\pi(Q(G))\pi(Q_{v}(H))+\pi(Q(H))\pi(Q_{u}(G)-x\pi(Q_{v}(H))\pi(Q_{u}(G).$
\end{lemma}

Let $B_{n}$ be the matrix of order $n$ obtained from $L(P_{n+1})$ by deleting the row and column corresponding to one end of
$P_{n+1}$, $U_{n}$ be the matrix of order $n$ obtained from $L(P_{n+2})$ by deleting the rows and columns corresponding to the two end vertices of $P_{n+2}$.

\begin{lemma}\label{art27}{\cite{liuwu1}}
Let $y$ be the root of the equation $y^{2}-(x-2)y-1=0$. Then\\
$(a)\pi(L(P_{n});y)=\frac{y^{2n}(y+1)^{2}+(-1)^{n}(y-1)^{2}}{y^{n}(y^{2}+1)},(n\geq 1);$\\
$(b)\pi(B_{n};y)=\frac{y^{2n+1}(y+1)-(-1)^{n}(y-1)}{y^{n}(y^{2}+1)},(n\geq 1);$\\
$(c)\pi(U_{n};y)=\frac{y^{2n+2}+(-1)^{n}}{y^{n}(y^{2}+1)},(n\geq 1);$\\
$(d)\pi(L(C_{n});y)=\frac{y^{2n}+(-1)^{n}}{y^{n}}+2,(n\geq 3).$
\end{lemma}

\begin{lemma}\label{art28}{\cite{liuwu2}}
 Suppose $y=1$, we have\\
$(a)\pi(L(P_{n});y=1)=2,(n\geq 1);$\\
$(b)\pi(B_{n};y=1)=1;$\\
$(c)\pi(U_{n};y=1)=\frac{1+(-1)^{n}}{2},(n\geq 1);$\\
$(d)\pi(L(C_{n});y=1)=(-1)^{n}+3,(n\geq 3).$
\end{lemma}

Let $B'_{n}$ be the matrix of order $n$ obtained from $Q(P_{n+1})$ by deleting the row and column corresponding to one end of
$P_{n+1}$, $U'_{n}$ be the matrix of order $n$ obtained from $Q(P_{n+2})$ by deleting the rows and columns corresponding to the two end vertices of $P_{n+2}$.

\begin{lemma}\label{art29}{\cite{liuwu1}}
Let $y$ be the root of the equation $y^{2}-(x-2)y-1=0$. Then\\
$(a)\pi(Q(P_{n});y)=\frac{y^{2n}(y+1)^{2}+(-1)^{n}(y-1)^{2}}{y^{n}(y^{2}+1)},(n\geq 1);$\\
$(b)\pi(B'_{n};y)=\frac{y^{2n+1}(y+1)-(-1)^{n}(y-1)}{y^{n}(y^{2}+1)},(n\geq 1);$\\
$(c)\pi(U'_{n};y)=\frac{y^{2n+2}+(-1)^{n}}{y^{n}(y^{2}+1)},(n\geq 1);$\\
$(d)\pi(Q(C_{n});y)=\frac{y^{2n}+(-1)^{n}}{y^{n}}+2(-1)^{n},(n\geq 3).$
\end{lemma}

\begin{lemma}\label{art210}{\cite{liuwu2}}
 Suppose $y=1$, we have\\
$(a)\pi(Q(P_{n});y=1)=2,(n\geq 1);$\\
$(b)\pi(B'_{n};y=1)=1;$\\
$(c)\pi(U'_{n};y=1)=\frac{1+(-1)^{n}}{2},(n\geq 1);$\\
$(d)\pi(Q(C_{n});y=1)=3(-1)^{n}+1,(n\geq 3).$
\end{lemma}

%\begin{lemma}\label{art211}{\cite{liuwu2}}
%If $y=1$,we have $\pi(Q(L_{p,q});y=1)=3(-1)^{p}+1$ ($\pi(L(L_{p,q});y=1)=(-1)^{p}+3$).
%\end{lemma}
Suppose that $G$ is a graph with vertex set $V(G)=\{v_{1}, v_{2},...,v_{n}\}.$ Let $K \subseteq \{1,...,n\}$ and denote by $H(K)$ the set of all subgraphs $H$ of $G$ having both of the following properties:

(1) The  vertices of $H$ are $\{v_{i}: i\in K\};$

(2) Each connected component of $H$ is either an edge or a cycle. For $H\in H(K),$ let $c(H)$ denote the number of components of $H$ which are cycles.

\begin{lemma}\label{art211}{\cite{bru}}
Let $G$ be a graph on $n$ vertices. If the degree sequence of $G$ is $(d_{1}, d_{2},...,d_{n}),$ then
$$per(L(G))=\sum\limits_{K\subseteq
\{1,...,n\}}\sum\limits_{H\in H(K)}(-1)^{|K|}2^{c(H)}\prod \limits_{i \notin K}d_{i}.$$
\end{lemma}

\begin{lemma}\label{art212}{\cite{liu}}
Let $G$ be a graph on $n$ vertices. If the degree sequence of $G$ is $(d_{1}, d_{2},...,d_{n}),$ then
$$per(Q(G))=\sum\limits_{K\subseteq
\{1,...,n\}}\sum\limits_{H\in H(K)}2^{c(H)}\prod \limits_{i \notin K}d_{i}.$$
\end{lemma}

A graph $G$ is nearly regular if the modulus of the difference between the degrees of any two vertices in $G$ does not exceed 1.

\begin{lemma}\label{art213}{\cite{liu}}
Let $G$ be a nearly regular  graph on $n$ vertices. If $\pi(L(H);x)=\pi(L(G);x)$ (or $\pi(Q(H);x)=\pi(Q(G);x)$), then $H$ has the same degree sequence as $G$.
\end{lemma}

\section{Proofs of Theorems \ref{art11} and \ref{art12}}
Before to give the proofs of Theorems \ref{art11} and \ref{art12}, we first calculate the Laplacian permanental polynomial of $d(p, q, r)$ and $\theta(p, q, r)$, respectively.

Regard $d(p, q, r)$ as the coalescence obtained from $C_{p}$ and $D_{q,r+1}$  by identifying the vertex $v$ of  $C_{p}$  with a pendant vertex of  $D_{q,r+1}$. Apply Lemma \ref{art26} $(a)$ to the coalesced vertex $v$, then apply Lemma \ref{art25} $(a)$ to the vertex $u$ of $D_{q,r+1}$, we get
\begin{eqnarray*}
\pi(L(d(p, q, r));x)&=&\pi(L(C_{p}))[(x-3)\pi(U_{q-1})\pi(U_{r})+2\pi(U_{q-2})\pi(U_{r})
+\pi(U_{q-1})\pi(U_{r-1})\\
&&+2\pi(U_{r})]+\pi(U_{p-1})[(x-3)\pi(U_{q-1})\pi(B_{r+1})+2\pi(U_{q-2})\pi(B_{r+1})
\\
&&+\pi(U_{q-1})\pi(B_{r})+2\pi(B_{r+1})]-x\pi(U_{p-1})[(x-3)\pi(U_{q-1})\pi(U_{r})\\
&&+2\pi(U_{q-2})\pi(U_{r})
+\pi(U_{q-1})\pi(U_{r-1})+2\pi(U_{r})].
\end{eqnarray*}
Plugging $x=\frac{y^{2}+2y-1}{y}$ and Lemma \ref{art27} $(b)$, $(c)$, and $(d)$, by Maple 13.0, we have
\begin{equation}\label{equ2}
  y^{p+q+r}(y^{2}+1)^{3}\pi(L(d(p, q, r));y)-f(y)=f_{L}(p, q, r;y),
\end{equation}
where $n=p+q+r$,
$f(y)=4y^{2n+2}+(-1)^{n}+2(-1)^ny+5(-1)^ny^{2}+4(-1)^{n}y^{4}+4(-1)^ny^{3}+y^{2n+6}+5y^{2n+4}-2y^{2n+5}-4y^{2n+3},$ and
\begin{align*}
\begin{array}{lllll}
f_{L}(p, q, r;y)&=2y^{2p+2r+q+8}&+2y^{2q+2r+p+6}&+(-1)^{r}y^{2+2p+2q}&-2(-1)^{r+p}y^{5+q}\\
&+2(-1)^{p+r}y^{3+q}&+2(-1)^{r}y^{2p+q+5}&-2(-1)^{r}y^{2p+q+3}&-2(-1)^{p}y^{2r+q+8}\\
&-2(-1)^{p}y^{2r+q+7}&+2(-1)^{p}y^{2r+q+5}&-4(-1)^{r}y^{2p+q+1}&+4(-1)^{p}y^{2r+q+3}\\
&+8(-1)^{r}y^{p+q+2}&+2(-1)^{p}y^{2r+2q+3}&+4(-1)^{r+p}y^{1+q}&+4(-1)^{p}y^{2r+q+4}\\
&+2(-1)^{r}y^{2p+q+2}&+6(-1)^{p+r}y^{q+2}&+(-1)^{p+q}y^{2r+6}&+2(-1)^{q}y^{2r+p+6}\\
&+2(-1)^{q+r}y^{p+1}&+(-1)^{p}y^{2r+2q+6}&(-1)^{q}y^{2p+2r+6}&+6(-1)^{q+r}y^{p+2}\\
&+2(-1)^{p}y^{q+2r+6}&+4(-1)^{q+r}y^{p+4}&+(-1)^{q+r}y^{2p+2}&-2(-1)^{p+r}y^{2q+3}\\
&-2(-1)^{r}y^{2q+p+3}&+(-1)^{p}y^{2q+2r+4}&-2(-1)^{r}y^{2q+p+1}&+(-1)^{r+p}y^{2q+2}\\
&+2(-1)^{r}y^{2q+p+2}&+2(-1)^{q+r}y^{p+3}&+2(-1)^{q}y^{2p+2r+3}&+2(-1)^{q}y^{2r+p+3}\\
&+2(-1)^{p+q}y^{2r+5}&+2(-1)^{q}y^{2r+p+5}&-2(-1)^{r}y^{2p+2q+1}&+6(-1)^{p+r}y^{q+4}\\
&+4(-1)^{r}y^{p+q+4}&-2(-1)^{q+r}y^{2p+3}&+(-1)^{q}y^{2p+2r+4}&+(-1)^{p+q}y^{2r+4}\\
&+2(-1)^{q}y^{2r+p+4}&-2(-1)^{r}y^{2p+q+4}&+(-1)^{r}y^{2p+2q}&+(-1)^{p+r}y^{p+r}y^{2q}\\
&+2(-1)^{r}y^{2q+p}&+(-1)^{q+r}y^{2p}+&2(-1)^{q+r}y^{p}&+4(-1)^{r}y^{2p+q}\\
&+4(-1)^{r}y^{p+q}&-4y^{2p+2r+q+3}&+4y^{2r+p+q+6}&+2y^{2p+2r+q+6}\\
&+4y^{2p+2r+q+4}&-2y^{2q+2r+p+5}&+4y^{2q+2r+p+2}&-2y^{2q+2r+p+3}\\
&+2y^{2p+2r+q+7}&-2y^{2p+2r+q+5}&+4y^{2p+2r+q+2}&+4y^{2r+p+q+2}\\
&+8y^{2r+p+q+4}&+6y^{2q+2r+p+4}.
\end{array}
\end{align*}

Substituting $y=1$ into Equation (\ref{equ2}), we get the following result.
\begin{lemma}\label{art31}
Let $y=1$, then
%\begin{eqnarray*}
$\pi(L(d(p, q, r));y=1)=\frac{9}{2}+\frac{3(-1)^{p}}{2}+\frac{3(-1)^{q}}{2}+\frac{(-1)^{p+q}}{2}+2(-1)^{q+r}+2(-1)^{r}+2(-1)^{p+r}+2(-1)^{p+q+r}.$
%\end{eqnarray*}
\end{lemma}
Correspondingly, apply Lemma \ref{art25} $(a)$ to the two vertices $v$ and $u$ of $\theta(p, q, r)$ respectively, we have
\begin{eqnarray*}
&&\pi(L(\theta(p, q, r));x)\\
&=&(x-3)[(x-3)\pi(U_{p})\pi(U_{q})\pi(U_{r})+\pi(U_{p-1})\pi(U_{q})\pi(U_{r})
+\pi(U_{p})\pi(U_{q-1})\pi(U_{r})\\
&&+\pi(U_{p})\pi(U_{q})\pi(U_{r-1})]+[(x-3)\pi(U_{p-1})\pi(U_{q})\pi(U_{r})+\pi(U_{p-2})\pi(U_{q})\pi(U_{r})\\
&&+\pi(U_{p-1})\pi(U_{q-1})\pi(U_{r})+\pi(U_{p-1})\pi(U_{q})\pi(U_{r-1})]+[(x-3)\pi(U_{p})\pi(U_{q-1})\pi(U_{r})\\
&&+\pi(U_{p-1})\pi(U_{q-1})\pi(U_{r})
+\pi(U_{p})\pi(U_{q-2})\pi(U_{r})+\pi(U_{p})\pi(U_{q-1})\pi(U_{r-1})]\\
&&+[(x-3)\pi(U_{p})\pi(U_{q})\pi(U_{r-1})+\pi(U_{p-1})\pi(U_{q})\pi(U_{r-1})
+\pi(U_{p})\pi(U_{q-1})\pi(U_{r-1})\\
&&+\pi(U_{p})\pi(U_{q})\pi(U_{r-2})]
+2\pi(U_{p})+2\pi(U_{q})+2\pi(U_{r}).
\end{eqnarray*}
Plugging $x=\frac{y^{2}+2y-1}{y}$ and Lemma \ref{art28} $(c)$, with the help of Maple 13.0, we obtain
\begin{equation}\label{equ3}
  y^{p+q+r+2}(y^{2}+1)^{3}\pi(L(\theta(p, q, r));y)-f(y)=f_{L}(p, q, r;y),
\end{equation}
where $n=p+q+r+2$,
$f(y)=4(-1)^{n-2}y^{4}+4(-1)^{n-2}y^{3}+2(-1)^{n-2}y+5(-1)^{n-2}y^{2}+(-1)^{n-2}+y^{2n+6}-2y^{2n+5}+5y^{2n+4}+y^{2n}+3y^{2n+2}-4y^{2n+3}$,  and
\begin{align*}
\begin{array}{lllll}
f_{L}(p, q, r;y)&=2y^{2p+q+r+8}&+2y^{2q+r+p+8}&+2y^{2r+q+p+8}&+(-1)^{q}y^{6+2p+2r}\\
&+(-1)^{r}y^{6+2p+2q}&+2(-1)^{p}y^{6+2q+2r}&+4y^{2p+r+q+6}&-2(-1)^{q}y^{2p+2r+5}\\
&+(-1)^{q+r}y^{2p+6}&-2(-1)^{r}y^{2p+2q+5}&+(-1)^{p+q}y^{2r+6}&-2(-1)^{p}y^{2r+2q+5}\\
&+(-1)^{r+p}y^{2q+6}&+(-1)^{q}y^{2p+2r+4}&+2(-1)^{r+q}y^{2p+5}&+2(-1)^{q}y^{p+r+6}\\
&+2(-1)^{r}y^{p+q+6}&+2(-1)^{p+q}y^{2r+5}&+2(-1)^{p}y^{q+r+6}&+2(-1)^{p+r}y^{2q+5}\\
&+(-1)^{q+r}y^{2p+4}&+(-1)^{p+q}y^{2r+4}&+2(-1)^{r+p}y^{4+2q}&+(-1)^{r}y^{2p+2q+2}\\
&+4(-1)^{q}y^{p+r+4}&+4(-1)^{r}y^{q+p+4}&+4(-1)^{p}y^{q+r+4}&-(-1)^{p+r}y^{2q+2}\\
&+2(-1)^{q}y^{2+p+r}&+2(-1)^{r}y^{q+p+2}&+2(-1)^{p}y^{2+q+r}&+2y^{2p+q+r+4}\\
&+2y^{2r+p+q+4}&+2y^{2q+p+r+4}&+4y^{2q+p+r+6}&+4y^{p+q+2r+6}.
\end{array}
\end{align*}
Substituting $y=1$ into Equation (\ref{equ3}), we have the following result immediately.
\begin{lemma}\label{art32}
Let $y=1$, then
%\begin{eqnarray*}
$\pi(L(\theta(p, q, r));y=1)=\frac{7}{2}+\frac{(-1)^{p+r}}{2}+\frac{(-1)^{p+q}}{2}+\frac{(-1)^{r+q}}{2}+(-1)^{q}+(-1)^{r}+(-1)^{p}+2(-1)^{p+q+r}.$
%\end{eqnarray*}
\end{lemma}

\begin{lemma}\label{art33}
There exist no two non-isomorphic  graphs as $d(p, q, r)$ share the same Laplacian permanental polynomial.
\end{lemma}
\begin{proof}
Suppose that $d(p', q', r')$ and $d(p, q, r)$ are two bicyclic graphs with vertices $n'=p'+q'+r'$, $n=p+q+r$,  respectively. If $d(p', q', r')$ and $d(p, q, r)$ have the same Laplacian permanental polynomials, then by (i) of Lemma \ref{art23}, we have
\begin{equation}\label{}
  p'+q'+r'=p+q+r,
\end{equation}
from (\ref{equ2}), we obtain that
\begin{equation}\label{}
  f_{L}(p',q',r';y)=f_{L}(p,q,r;y).
\end{equation}
Obviously, the term in $f_{L}(p,q,r;y)$ with the largest exponent is $2y^{2p+2r+q+8},$ $2y^{2q+2r+p+6}$ or $(-1)^{r}y^{2+2p+2q},$ and similarly for $f_{L}(p',q',r';y)$. From the equations above, we have $2y^{2p'+2r'+q'+8}\\=2y^{2p+2r+q+8},$  $2y^{2q'+2r'+p'+6}=2y^{2q+2r+p+6}$ or $(-1)^{r'}y^{2+2p'+2q'}=(-1)^{r}y^{2+2p+2q}$.

For $2y^{2p'+2r'+q'+8}=2y^{2p+2r+q+8}$, we have $p=r'$, $r=p'$, $q=q'$ or $p=p'$, $q=q'$, $r=r'$. The same for $2y^{2q'+2r'+p'+6}=2y^{2q+2r+p+6}$, we get $q=r$', $r=q'$, $p=p'$ or $p=p'$, $q=q'$, $r=r'$. It's obvious that we can't have  $p=r'$, $r=p'$, $q=q'$ or $q=r'$, $r=q'$, $p=p'$ by Lemma \ref{art211}.
Conclude above, we have $p=p'$, $q=q'$, $r=r'$ or $p=q'$, $q=p'$, $r=r'$, that is $d(p', q', r')\cong d(p, q, r)$.
\end{proof}

%\begin{theorem}\label{art34}
%$d(p, q, r)$ is determined by the Laplacian permanental polynomial.
%\end{theorem}
{\bf Proof of Theorem \ref{art11}.}
Let $G$ be a graph having the same Laplacian permanental polynomial as $d(p, q, r)$. By the definition of $d(p, q, r)$ and Lemma \ref{art213}, we get that the degree sequence of $G$ is $(3^{2}, 2^{n-2})$.

Assume that $G$ is a connected graph. By Lemma \ref{art33}, we obtain that $G\cong d(p, q, r)$.

When $G$ is disconnected, we consider the following two cases.

If  $G=d(p', q', r')\cup(\bigcup\limits_{i=1}^{k} C_{k_{i}})$, then by Lemmas  \ref{art28}$(d)$ and \ref{art31}, $\pi(L(G);y=1)\neq \pi(L(d(p, q, r));\\y=1)=\frac{9}{2}+\frac{3(-1)^{p}}{2}+
\frac{3(-1)^{q}}{2}+\frac{(-1)^{p+q}}{2}+2(-1)^{q+r}+2(-1)^{r}+2(-1)^{p+r}+2(-1)^{p+q+r}$. This contradicts
the  assumption that $G$ and $d(p, q, r)$ are Laplacian copermanental.

If  $G=\theta(p', q', r')\cup(\bigcup\limits_{i=1}^{k} C_{k_{i}})$, then by Lemmas \ref{art28}$(d)$, \ref{art31} and \ref{art32}, $\pi(L(G);y=1)\neq \pi(L(d(p, q, r));y=1)=\frac{9}{2}+\frac{3(-1)^{p}}{2}+
\frac{3(-1)^{q}}{2}+\frac{(-1)^{p+q}}{2}+2(-1)^{q+r}+2(-1)^{r}+2(-1)^{p+r}+2(-1)^{p+q+r}$, a contradiction.
$\Box$

\begin{lemma}\label{art34}
There exist no two non-isomorphic graphs as $\theta(p, q, r)$ share the same Laplacian permanental polynomial.
\end{lemma}
\begin{proof}
Suppose $\theta(p', q', r')$ and $\theta(p, q, r)$ be two bicyclic graphs with vertices $n'=p'+q'+r'+2, n=p+q+r+2$, respectively. If $\theta(p', q', r')$ and $\theta(p, q, r)$ have the same Laplacian permanental polynomials, then by (i) of Lemma \ref{art23}, we have
\begin{equation}\label{}
  p'+q'+r'=p+q+r,
\end{equation}
by (\ref{equ3}), we have that
\begin{equation}\label{}
  f_{L}(p',q',r';y)=f_{L}(p,q,r;y).
\end{equation}
It's obvious that the term in $f_{L}(p,q,r;y)$ with the largest exponent is $2y^{2r+p+q+8}$, $2y^{2p+q+r+8}$, $2y^{2q+p+r+8}$, $(-1)^{q}y^{2p+2r+6}$, $(-1)^{p}y^{6+2q+2r}$ or $(-1)^{r}y^{6+2p+2q}$, and similarly for $f_{L}(p',q',r';y)$.
From the equations above,we have $2y^{2r+p+q+8}=2y^{2r'+p'+q'+8}$, $2y^{2p+q+r+8}=2y^{2p'+q'+r'+8}$, $2y^{2q+p+r+8}=2y^{2q'+p'+r'+8}$, $(-1)^{q}y^{2p+2r+6}=(-1)^{q'}y^{2p'+2r'+6}$, $(-1)^{p}y^{6+2q+2r}=(-1)^{p'}y^{6+2q'+2r'}$
or $(-1)^{r}y^{6+2p+2q}=(-1)^{r'}y^{6+2p'+2q'}$. By Lemma \ref{art211}, for all the
cases above, we have $p=p'$, $q=q'$, $r=r'$, i.e., $\theta(p', q', r')$ is isomorphic to $\theta(p, q, r)$.
\end{proof}

%\begin{theorem}\label{art36}
%$\theta(p, q, r)$ is determined by the Laplacian permanental polynomial.
%\end{theorem}
{\bf Proof of Theorem \ref{art12}.}
Let $G$ be a graph sharing the same Laplacian permanental polynomial as $\theta(p, q, r)$. Obtained from the definition of $\theta(p, q, r)$ and Lemma \ref{art213}, we get the degree sequence of $G$ is $(3^{2}, 2^{n-2})$.

It's obvious that $G\cong \theta(p, q, r)$ when $G$ is connected by Lemma \ref{art34}.

If $G=\theta(p', q', r')\cup(\bigcup\limits_{i=1}^{k} C_{k_{i}})$, then by Lemma \ref{art28} $(d)$ and Lemma \ref{art32}, $\pi(L(G);y=1)\neq \pi(L(\theta(p, q, r));y=1)=\frac{7}{2}+\frac{(-1)^{p+r}}{2}+
\frac{(-1)^{p+q}}{2}+\frac{(-1)^{r+q}}{2}+(-1)^{q}+(-1)^{r}+(-1)^{p}+2(-1)^{p+q+r}$. This contradicts
the  assumption that $G$ and $\theta(p, q, r)$ are Laplacian copermanental.

If $G=d(p', q', r')\cup(\bigcup\limits_{i=1}^{k} C_{k_{i}})$, by Lemma  \ref{art28} $(d)$ and Lemmas \ref{art31}, \ref{art32}, we have $\pi(L(G);y=1)\neq \pi(L(\theta(p, q, r));y=1)=\frac{7}{2}+\frac{(-1)^{p+r}}{2}+
\frac{(-1)^{p+q}}{2}+\frac{(-1)^{r+q}}{2}+(-1)^{q}+(-1)^{r}+(-1)^{p}+2(-1)^{p+q+r}$, a contradiction.
$\Box$

\section{Proofs of Theorems \ref{art13} and \ref{art14}}
Similar to the proofs of Theorems \ref{art11} and \ref{art12}, we compute the signless Laplacian permanental of $d(p, q, r)$ and $\theta(p, q, r)$ at first.

By the definition of $d(p, q, r)$ and Lemmas \ref{art25} $(b)$ and \ref{art26} $(b)$, we get
\begin{eqnarray*}
\pi(Q(d(p, q, r));x)&=&\pi(Q(C_{p}))[(x-3)\pi(U'_{q-1})\pi(U'_{r})+2\pi(U'_{q-2})\pi(U'_{r})
+\pi(U'_{q-1})\pi(U'_{r-1})\\
&&+2(-1)^{q}\pi(U'_{r})]
+\pi(U'_{p-1})[(x-3)\pi(U'_{q-1})\pi(B'_{r+1})+2\pi(U'_{q-2})\pi(B'_{r+1})\\
&&+\pi(U'_{q-1})\pi(B'_{r})
+2(-1)^{q}\pi(B'_{r+1})]-x\pi(U'_{p-1})[(x-3)\pi(U'_{q-1})\pi(U'_{r})\\
&&+2\pi(U'_{q-2})\pi(U'_{r})
+\pi(U'_{q-1})\pi(U'_{r-1})+2(-1)^{q}\pi(U'_{r})].
\end{eqnarray*}
Plugging $x=\frac{y^{2}+2y-1}{y}$ and Lemma \ref{art29} $(b)$, $(c)$, and $(d)$, with the help of Maple 13.0, we have
\begin{equation}\label{equ4}
  y^{p+q+r}(y^{2}+1)^{3}\pi(Q(d(p, q, r));y)-f(y)=f_{Q}(p, q, r;y),
\end{equation}
where $n=p+q+r$, $f(y)=4(-1)^{n}y^{4}+4(-1)^{n}y^{3}+2(-1)^ny+5(-1)^ny^{2}-4y^{2n+3}
+4y^{2n+2}+y^{2n+6}+5y^{2n+4}-2y^{2n+5}+(-1)^{n}$ and
\begin{align*}
\begin{array}{llllll}
&f_{Q}(p, q, r;y)\\
&=2(-1)^{q}y^{2p+2r+q+8}&+2(-1)^{p}y^{2q+2r+p+6}&+(-1)^{r}y^{2+2p+2q}&-2(-1)^{q+r+p}y^{5+q}\\
&+2(-1)^{q}y^{2p+2r+q+7}&+2(-1)^{q+r}y^{2p+q+5}&-2(-1)^{p+q}y^{2r+q+8}
&-2(-1)^{p+q}y^{2r+q+7}\\&+4(-1)^{q+r}y^{2p+q}&+2(-1)^{p+r}y^{2q+p}&+2(-1)^{p+q+r}y^{p}
&+(-1)^{r}y^{2p+2q}\\
&-2(-1)^{q+r}y^{2p+q+4}&+(-1)^{q+r}y^{2p+2}&+(-1)^{r+p}y^{2+2q}&-2(-1)^{r}y^{2p+2q+1}\\
&+2(-1)^{p}y^{2q+2r+3}&-2(-1)^{q+r}y^{2p+3}&-2(-1)^{p+r}y^{2q+3}&+2(-1)^{p+q}y^{2r+5}\\
&+(-1)^{q}y^{2p+2r+4}&+(-1)^{p}y^{2r+2q+4}&+(-1)^{p+q}y^{4+2r}&+2(-1)^{q}y^{2p+2r+3}\\
&+(-1)^{q}y^{2p+2r+6}&+(-1)^{p}y^{2q+2r+6}&+(-1)^{q+p}y^{2r+6}&+4(-1)^{p+q}y^{2r+p+q+6}
\end{array}
\end{align*}
\begin{align*}
\begin{array}{lllll}
&+8(-1)^{p+q}y^{2r+p+q+4}&+4(-1)^{p+q+r}y^{p+q+4}&+4(-1)^{p+q}y^{2r+p+q+2}&+8(-1)^{q+r+p}y^{p+q+2}\\
&+4(-1)^{p+q+r}y^{1+q}&+2(-1)^{q}y^{2p+2r+q+6}&+2(-1)^{p+q}y^{p+2r+6}&+2(-1)^{p+q}y^{2r+q+6}\\
&+6(-1)^{p+q+r}y^{p+2}&+6(-1)^{p+q+r}y^{2+q}&-4(-1)^{q+r}y^{2p+q+1}&-2(-1)^{p+r}y^{2q+p+1}\\
&+4(-1)^{r+p+q}y^{p+q}&+2(-1)^{p+q+r}y^{p+1}&+4(-1)^{q}y^{2p+2r+q+2}&+4(-1)^{p}y^{2r+2q+p+2}\\
&+2(-1)^{p+q+r}y^{3+p}&+2(-1)^{p+q+r}y^{q+3}&+2(-1)^{q+r}y^{2p+q+2}&+2(-1)^{p+r}y^{p+2q+2}\\
&+4(-1)^{p+q+r}y^{4+p}&+6(-1)^{q+r+p}y^{q+4}&-2(-1)^{q+r}y^{2p+q+3}&-2(-1)^{p+r}y^{2q+p+3}\\
&+2(-1)^{p+q}y^{2r+p+3}&+4(-1)^{p+q}y^{2r+q+3}&+6(-1)^{p}y^{2r+p+2q+4}&+2(-1)^{p+q}y^{p+2r+4}\\
&+4(-1)^{p+q}y^{2r+q+4}&-4(-1)^{q}y^{2p+2r+q+3}&-2(-1)^{p}y^{2q+2r+p+3}&+(-1)^{p+r}y^{2q}\\
&-2(-1)^{q}y^{2p+2r+q+5}&-2(-1)^{p}y^{2q+2r+p+5}&+2(-1)^{p+q}y^{p+2r+5}&+2(-1)^{p+q}y^{2r+q+5}\\
&+4(-1)^{q}y^{2r+2p+q+4}&+(-1)^{q+r}y^{2p}.
\end{array}
\end{align*}
Substituting $y=1$ into Equation (\ref{equ4}), it's easy to get the following result.
\begin{lemma}\label{art35}
Let $y=1$, then $\pi(Q(d(p, q, r));y=1)=\frac{1}{2}+\frac{3(-1)^{p}}{2}+\frac{3(-1)^{q}}{2}+\frac{9(-1)^{p+q}}{2}+8(-1)^{p+q+r}$.
\end{lemma}

\begin{lemma}\label{art36}
There exist no two non-isomorphic  graphs as $d(p, q, r)$ share the same singless Laplacian permanental polynomial.
\end{lemma}

\begin{proof}
Let $d(p', q', r')$ and $d(p, q, r)$ be two bicyclic graphs with vertices $n'=p'+q'+r'$, $n=p+q+r$, respectively. If $d(p', q', r')$ and $d(p, q, r)$ have the same signless Laplacian permanental polynomials, then by (i) of Lemma \ref{art24}, we have
\begin{equation}\label{}
  p'+q'+r'=p+q+r,
\end{equation}
from (\ref{equ4}), we get that
\begin{equation}\label{}
  f_{Q}(p',q',r';y)=f_{Q}(p,q,r;y).
\end{equation}
Clearly, the term in $f_{Q}(p,q,r;y)$ with the largest exponent is $2y^{2p+2r+q+8}$, $2y^{2q+2r+p+6}$ or $(-1)^{r}y^{2+2p+2q}$, and similarly for $f_{Q}(p',q',r';y)$.
From the equations above,we have $2y^{2p'+2r'+q'+8}\\=2y^{2p+2r+q+8}$, $2y^{2q'+2r'+p'+6}=2y^{2q+2r+p+6}$ or $(-1)^{r'}y^{2+2p'+2q'}=(-1)^{r}y^{2+2p+2q}$. Obtained from Lemma \ref{art212}, we have $p=p'$, $q=q'$, $r=r'$ or $p=q'$, $q=p'$, $r=r'$. Then $d(p', q', r')$ is isomorphic to $d(p, q, r)$.
\end{proof}
Analogously, by the definition of $\theta(p, q, r)$ and Lemma \ref{art25} $(b)$, we have
\begin{eqnarray*}
&&\pi(Q(\theta(p, q, r));x)\\
&=&(x-3)[(x-3)\pi(U'_{p})\pi(U'_{q})\pi(U'_{r})+\pi(U'_{p-1})\pi(U'_{q})\pi(U'_{r})
+\pi(U'_{p})\pi(U'_{q-1})\pi(U'_{r})\\
&&+\pi(U'_{p})\pi(U'_{q})\pi(U'_{r-1})]+[(x-3)\pi(U'_{p-1})\pi(U'_{q})\pi(U'_{r})+\pi(U'_{p-2})\pi(U'_{q})\pi(U'_{r})\\
&&+\pi(U'_{p-1})\pi(U'_{q-1})\pi(U'_{r})+\pi(U'_{p-1})\pi(U'_{q})\pi(U'_{r-1})]+[(x-3)\pi(U'_{p})\pi(U'_{q-1})\pi(U'_{r})\\
&&+\pi(U'_{p-1})\pi(U'_{q-1})\pi(U'_{r})
+\pi(U'_{p})\pi(U'_{q-2})\pi(U'_{r})
+\pi(U'_{p})\pi(U'_{q-1})\pi(U'_{r-1})]\\
&&+[(x-3)\pi(U'_{p})\pi(U'_{q})\pi(U'_{r-1})+\pi(U'_{p-1})\pi(U'_{q})\pi(U'_{r-1})
+\pi(U'_{p})\pi(U'_{q-1})\pi(U'_{r-1})\\
&&+\pi(U'_{p})\pi(U'_{q})\pi(U'_{r-2})]+2(-1)^{q+r+2}\pi(U'_{p})+2(-1)^{p+r+2}\pi(U'_{q})+2(-1)^{p+q+2}\pi(U'_{r}).
\end{eqnarray*}
Plugging $x=\frac{y^{2}+2y-1}{y}$ and Lemma \ref{art210} $(c)$, with Maple 13.0, we obtain
\begin{equation}\label{equ5}
  y^{p+q+r+2}(y^{2}+1)^{3}\pi(Q(\theta(p, q, r));y)-f(y)=f_{Q}(p, q, r;y),
\end{equation}
where $n=p+q+r+2$,
$f(y)=4(-1)^{n-2}y^{4}+4(-1)^{n-2}y^{3}+2(-1)^{n-2}y+5(-1)^{n-2}y^{2}+(-1)^{n-2}-4y^{2n+3}-2y^{2n+5}+5y^{2n+4}+y^{2n}+10y^{2n+6}+3y^{2n+2}$, and
\begin{align*}
\begin{array}{lllll}
f_{Q}(p, q, r;y)&=2(-1)^{q+r}y^{2p+r+q+8}&+2(-1)^{p+r}y^{2q+r+p+8}&+2(-1)^{p+q}y^{2r+q+p+8}\\
&+(-1)^{s}y^{6+2p+2r}&+(-1)^{r}y^{6+2p+2q}&+2(-1)^{p}y^{6+2q+2r}\\
&+(-1)^{q}y^{2p+2r+4}&+(-1)^{p+q}y^{4+2r}&+2(-1)^{p+r}y^{2q+4}\\
&-2(-1)^{q}y^{2p+2r+5}&-2(-1)^{r}y^{2p+2q+5}&+2(-1)^{p+q}y^{2r+5}\\
&+2(-1)^{r+p}y^{2q+5}&+2(-1)^{q+r}y^{2p+5}&+(-1)^{p+q}y^{2r+6}\\
&+(-1)^{p+r}y^{2q+6}&-2(-1)^{p}y^{2q+2r+5}&+(-1)^{r+q}y^{2p+6}\\
&+(-1)^{r}y^{2p+2q+2}&-(-1)^{p+r}y^{2q+2}&+2(-1)^{p+q+r}y^{r+p+2}\\
&+2(-1)^{q+p+r}y^{q+r+2}&+(-1)^{r+q}y^{4+2p}&+2(-1)^{q+r}y^{2p+q+r+4}\\
&+4(-1)^{p+q+r}y^{p+q+4}&+4(-1)^{p+q+r}y^{r+p+4}&+4(-1)^{p+q+r}y^{q+r+4}\\
&+2(-1)^{p+r+q}y^{p+q+2}&+2(-1)^{p+q+r}y^{6+p+q}&+2(-1)^{p+q+r}y^{r+p+6}\\
&+2(-1)^{p+r+q}y^{6+q+r}&+2(-1)^{p+q}y^{2r+q+p+4}&+2(-1)^{p+r}y^{2q+p+r+4}\\
&+4(-1)^{p+q}y^{2r+p+q+6}&+4(-1)^{p+r}y^{2q+p+r+6}&+4(-1)^{q+r}y^{r+q+2p+6}.
\end{array}
\end{align*}
Substituting $y=1$ into Equation (\ref{equ5}), we know the following result.
\begin{lemma}\label{art37}
Let $y=1$, then $\pi(Q(\theta(p, q, r));y=1)=\frac{1}{2}+\frac{3(-1)^{p+r}}{2}+\frac{3(-1)^{p+q}}{2}+\frac{3(-1)^{r+q}}{2}+5(-1)^{p+q+r}$.
\end{lemma}

\begin{lemma}\label{art38}
There exist no two non-isomorphic graphs as $\theta(p, q, r)$ share the same signless Laplacian permanental polynomial.
\end{lemma}

\begin{proof}
Suppose $\theta(p', q', r')$ and $\theta(p, q, r)$ be two bicyclic graphs with vertices $n'=p'+q'+r'+2,n=p+q+r+2$, respectively. If $\theta(p', q', r')$ and $\theta(p, q, r)$ have the same signless Laplacian permanental polynomials, then by (i) of Lemma \ref{art24}, we have
\begin{equation}\label{}
  p'+q'+r'=p+q+r,
\end{equation}
obtained from (\ref{equ5}), we know that
\begin{equation}\label{}
  f_{Q}(p',q',r';y)=f_{Q}(p,q,r;y).
\end{equation}
Obviously, the term in $f_{Q}(p,q,r;y)$ with the largest exponent is $2y^{2r+p+q+8}$, $2y^{2p+q+r+8}$, $2y^{2q+p+r+8}$, $(-1)^{q}y^{2p+2r+6}$, $(-1)^{p}y^{6+2q+2r}$ or $(-1)^{r}y^{6+2p+2q}$, and similarly for $f_{Q}(p',q',r';y)$.
From the equations above, we have $2y^{2r+p+q+8}=2y^{2r'+p'+q'+8}$, $2y^{2p+q+r+8}=2y^{2p'+q'+r'+8}$, $2y^{2q+p+r+8}=2y^{2q'+p'+r'+8}$, $(-1)^{q}y^{2p+2r+6}=(-1)^{q'}y^{2p'+2r'+6}$, $(-1)^{p}y^{6+2q+2r}=(-1)^{p'}y^{6+2q'+2r'}$
or $(-1)^{r}y^{6+2p+2q}\\=(-1)^{r'}y^{6+2p'+2q'}$. All the
cases above with Lemma \ref{art212}, we obtain that $\theta(p', q', r')$ is isomorphic to $\theta(p, q, r)$.
\end{proof}

%\begin{theorem}\label{art311}
%$d(p, q, r)$ is determined by the singless Laplacian permanental polynomial.
%\end{theorem}

{\bf Proof of Theorem \ref{art13}.}
Let $G$ be a graph having the same signless Laplacian permanental polynomial as $d(p, q, r)$. By the definition of $d(p, q, r)$ and Lemma \ref{art213}, we know that the degree sequence of $G$ is $(3^{2}, 2^{n-2})$.

By Lemma \ref{art36} we have that $G\cong d(p, q, r)$ when $G$ is connected.

If $G=d(p', q', r')\cup(\bigcup\limits_{i=1}^{k} C_{k_{i}})$, then by Lemmas \ref{art210}$(d)$ and \ref{art35}, $\pi(Q(G);y=1)\neq \pi(Q(d(p, q, r));\\y=1)=\frac{1}{2}+\frac{3(-1)^{p}}{2}+
\frac{3(-1)^{q}}{2}+\frac{9(-1)^{p+q}}{2}+8(-1)^{p+q+r}$. This contradicts
the  assumption that $G$ and $d(p, q, r)$ are signless Laplacian copermanental.

If $G=\theta(p', q', r')\cup(\bigcup\limits_{i=1}^{k} C_{k_{i}})$, then by Lemma \ref{art210} $(d)$ and Lemmas \ref{art35}, \ref{art37}, $\pi(L(G);y=1)\neq \pi(Q(d(p, q, r));y=1)=\frac{1}{2}+\frac{3(-1)^{p}}{2}+
\frac{3(-1)^{q}}{2}+\frac{9(-1)^{p+q}}{2}+8(-1)^{p+q+r}$, a contradiction.
$\Box$

%\begin{theorem}\label{art312}
%$\theta(p, q, r)$ is determined by the signless Laplacian permanental polynomial.
%\end{theorem}

{\bf Proof of Theorem \ref{art14}.}
Let $G$ be a graph with the same signless Laplacian permanental polynomial as $\theta(p, q, r)$. Analogously, by the definition of $\theta(p, q, r)$ and Lemma \ref{art213}, we know that the degree sequence of $G$ is $(3^{2}, 2^{n-2})$.

By Lemma \ref{art38}, we know that $G$ is isomorphic to $\theta(p, q, r)$ if $G$ is connected.

If $G=\theta(p', q', r')\cup(\bigcup\limits_{i=1}^{k} C_{k_{i}})$, then by Lemmas \ref{art210}$(d)$ and \ref{art37}, $\pi(Q(G);y=1)\neq \pi(Q(\theta(p, q, r));\\y=1)=\frac{1}{2}+\frac{3(-1)^{p+r}}{2}+
\frac{3(-1)^{p+q}}{2}+\frac{3(-1)^{r+q}}{2}+5(-1)^{p+q+r}$. This contradicts
the  assumption that $G$ and $\theta(p, q, r)$ are signless Laplacian copermanental.

If $G=d(p', q', r')\cup(\bigcup\limits_{i=1}^{k} C_{k_{i}})$, then by Lemma \ref{art210} $(d)$ and Lemmas \ref{art35}, \ref{art37}, $\pi(Q(G);y=1)\neq \pi(Q(\theta(p, q, r));y=1)=\frac{1}{2}+\frac{3(-1)^{p+r}}{2}+\frac{3(-1)^{p+q}}{2}+
\frac{3(-1)^{r+q}}{2}+5(-1)^{p+q+r}$, a contradiction.
$\Box$

\noindent{\bf Data Availability}

No data were used to support this study.

\noindent{\bf Conflicts of Interest}

The authors declare that they have no conflicts of interest.

\noindent{\bf Acknowledgement}

The frist author is supported by the
NSFC Grant (11761056) and the NSF
of Qinghai Province Grant(2020-ZJ-920).


\begin{thebibliography}{8}


\bibitem{bap}
R. Bapat, A bound for the permanent of the Laplacian matrix, \textit{Linear Algebra Appl.} 74 (1986) 219--223.

\bibitem{bot}
P. Botti, R. Merris, C. Vega, Laplacian permanents of trees, \textit{SIAM J. Discrete Math.} 5 (1992) 460--466.


\bibitem{bru}
R. Brualdi, J. Goldwasser, Permanent of the Laplacian matrix of trees and bipartite graphs, \textit{Discrete Math.} 48 (1984) 1--21.



%\bibitem{bru2}
%R. A. Brualdi, D. $Cvetkovi\acute{c}$, A Combinatorial Approach to Matrix Theory and its Applications, CRC press, 2008.

\bibitem{cas}
G. Cash, I. Gutman, The Laplacian permanental polynomial: formulas and algorithms, \textit{MATCH
Commun. Math. Comput. Chem.} 51 (2004) 129--136.



\bibitem{far}
I. Faria, Permanental roots and the star degree of a graph, \textit{Linear Algebra Appl.} 64 (1985) 255--265.

\bibitem{far2}
I. Faria, Multiplicity of integer roots of polynomials of graphs, \textit{Linear Algebra Appl.} 229 (1995) 15--35.






\bibitem{gen}
X. Geng, X. Hu, S. Li, Further results on permanental bounds for the Laplacian matrix of trees,
\textit{Linear Multilinear Algebra} 58 (2010) 571--587.

\bibitem{gen2}
X. Geng, X. Hu, S. Li, Permanental bounds of the Laplacian matrix of trees with given domination
number. \textit{Graph Combin.} 31 (2015) 1423--1436.

\bibitem{gol}
J.  Goldwasser, Permanent of the Laplacian matrix of trees with a given matching, \textit{Discrete Math.} 61 (1986) 197--212.


\bibitem{lis}
S. Li, Y. Li, X. Zhang, Edge-grafting theorems on permanents of the Laplacian matrices of graphs
and their applications, \textit{Electron. J. Linear Algebra} 26 (2013) 28--48.

\bibitem{lizh}
S. Li, L. Zhang, Permanental bounds for the signless Laplacian matrix of bipartite graphs and unicyclic
graphs, \textit{Linear Multilinear Algebra} 59 (2011) 145--158.

\bibitem{lizh2}
S. Li, L. Zhang, Permanental bounds for the signless Laplacian matrix of a unicyclic graph with
diameter $d$, \textit{Graphs Combin.} 28 (2012) 531--546.

\bibitem{liuwu1}
X. Liu, T. Wu, Computing the permanental polynomials of graphs, \textit{Appl. Math.  Comput.} 304 (2017) 103--113.


\bibitem{mer}
R. Merris, K.R. Rebman, W. Watkins, Permanental polynomials of graphs, \textit{Linear Algebra Appl.} 38 (1981) 273--288.

\bibitem{mer2}
R. Merris, The Laplacian permanental polynomial for trees, \textit{Czech. Math. J.} 32 (1982) 397--403.

%\bibitem{min}
%H. Minc, Permanents, Cambridge University Press, 1984.



\bibitem{liuwu2}
X. Liu, T. Wu, Graphs determined by the (signless) Laplacian permanental polynomial, \textit{Linear  Multilinear Algebra}, 2020, https://doi.org/10.1080/03081087.2020.1849003.



\bibitem{liu}
S. Liu, On the (signless) Laplacian permanental polynomials of graphs, \textit{Graphs  Combin.} 35 (2019) 787--803.

\bibitem{wut}
T. Wu, T. Zhou, H. L\"{u}, Further results on the star degree of graphs, \textit{Appl. Math.  Comput.} 2022, https://doi.org/10.1016/j.amc.2022.127076.

\bibitem{val}
L. Valiant, The complexity of computing the permanent, \textit{Theoret. Comput. Sci.}  8 (1979) 189--201.


\bibitem{vrb}
A. Vrba, Principal subpermanents of the Laplacian matrix, \textit{Linear Multilinear Algebra} 19 (1986) 335--346.


\end{thebibliography}
\end{document}